\documentclass[a4paper,11pt]{amsart}

\usepackage{graphicx, rotating}
\usepackage{subfigure}

\usepackage[utf8]{inputenc}
\usepackage[toc]{appendix}
\usepackage{twoopt}

\usepackage{amssymb}
\usepackage{amsthm}
\usepackage{amsmath,a4wide}
\numberwithin{equation}{section}
\usepackage{hyperref}
\usepackage{mathtools}
\usepackage{amsfonts}
\usepackage{amsaddr}

\usepackage{marvosym}
\usepackage{verbatim}
\usepackage{upgreek}

\usepackage{enumerate}
\usepackage{pdfsync}
\usepackage{caption}
\usepackage{hyperref}
\usepackage{enumerate}
\mathtoolsset{showonlyrefs}

\theoremstyle{plain}
\newtheorem{theorem}{Theorem}[section]
\theoremstyle{plain}
\newtheorem{corollary}[theorem]{Corollary}
\theoremstyle{plain}
\newtheorem{lemma}[theorem]{Lemma}

\theoremstyle{plain}
\newtheorem{proposition}[theorem]{Proposition}
\theoremstyle{definition}

\theoremstyle{remark}

\theoremstyle{remark}

\theoremstyle{definition}

\theoremstyle{plain}
\newtheorem{conjecture}[theorem]{Conjecture}
\theoremstyle{definition}

\newcommand{\R}{\mathbb{R}}
\newcommand{\C}{\mathbb{C}}

\newcommand{\Z}{\mathbb{Z}}
\newcommand{\N}{\mathbb{N}}
\newcommand{\T}{\mathbb{T}}

\renewcommand{\l}{\lambda}
\renewcommand{\L}{\Lambda}

\newcommand{\D}{\mathbb{D}}
\newcommand{\LC}{\mathcal{L}}

\newcommandtwoopt{\xarrow}[2][0.5cm][0]{\mathrel{\rotatebox[origin=c]{#2}{$\xrightarrow{\rule{#1}{0pt}}$}}}

\begin{document}

\title[Hypergeometric Functions, Heat Kernels and a ``Weltkonstante"]{
		An Application of Hypergeometric Functions to Heat Kernels on Rectangular and Hexagonal Tori and a ``Weltkonstante"\\- Or -\\How Ramanujan Split Temperatures
	}

\author[Markus Faulhuber]{Markus Faulhuber}
\address{NuHAG, Faculty of Mathematics, University of Vienna\\ Oskar-Morgenstern-Platz 1, 1090 Vienna, Austria}
\email{markus.faulhuber@univie.ac.at}
\thanks{The author was supported by an Erwin--Schrödinger Fellowship of the Austrian Science Fund (FWF): J4100-N32. The results have partially been established during the first stage of the fellowship, which the author spent with the Analysis Group at NTNU Trondheim, Norway. The author wishes to thank Martin Ehler and Franz Luef for beneficial feedback on the first draft of the manuscript. The author wishes to thank the anonymous referees for a careful reading of the manuscript and suggestions for improvement of the manuscript, in particular for the encouragement to expand the section on the hexagonal torus.}

\begin{abstract}
	In this work we investigate the heat kernel of the Laplace--Beltrami operator on a rectangular torus and the according temperature distribution. We compute the minimum and the maximum of the temperature on rectangular tori of fixed area by means of Gauss' hypergeometric function $_2F_1$ and the elliptic modulus. In order to be able to do this, we employ a beautiful result of Ramanujan, connecting hypergeometric functions, the elliptic modulus and theta functions. Also, we investigate the temperature distribution of the heat kernel on hexagonal tori and use Ramanujan's corresponding theory of signature 3 to derive analogous results to the rectangular case. Lastly, we show connections to the problem of finding the exact value of Landau's ``Weltkonstante", a universal constant arising in the theory of extremal holomorphic mappings; and for a related, restricted extremal problem we show that the conjectured solution is the second lemniscate constant.
\end{abstract}


\maketitle

\section{Introduction}\label{sec_Intro}
This article is inspired by a problem posed and investigated in the article of Baernstein, Eremenko, Fryntov and Solynin \cite{Baernstein_Metric_2005} and, again, in Eremenko's preprint \cite{Eremenko_Hyperbolic_2011}. The problem discussed in \cite{Baernstein_Metric_2005} and \cite{Eremenko_Hyperbolic_2011} is about finding the exact value of a constant, closely related to a universal constant arising from a problem in geometric function theory, posed by Landau \cite{Lan29}. As repeatedly suggested by Baernstein \cite{Baernstein_ExtremalProblems_1994}, \cite{Baernstein_HeatKernel_1997} and in the joint work of Baernstein and Vinson \cite{BaernsteinVinson_Local_1998}, we study the heat kernel on a torus, hoping to get a better insight to Landau's problem. However, we do not follow the suggestion to study heat kernels on tori with fixed covering radius, but rather tori of fixed surface area. Also, we solely focus on tori identifiable with rectangular lattices or a hexagonal lattice. This is due to the fact that only for these tori the location of the minimal temperature is time-independent. It is the high symmetry of the underlying rectangular and hexagonal lattices which forces the location of the minimal temperature to be stationary.

In \cite{Eremenko_Hyperbolic_2011}, Eremenko studies Landau's problem for rectangular lattices with lattice parameters $2 \omega$ and $2 \omega'$ and the restriction that the covering radius is fixed  to $4 \omega^2 + 4 \omega'^2 = 1$. This leads to the main idea in this work, which is to study the heat kernel as a function of the elliptic modulus $k$ and the complementary elliptic modulus $k'$ with $k^2+k'^2 = 1$. The results in this work may open a promising connection between Landau's problem and a minimum problem for the heat kernel on tori of fixed surface area.

We start with a theorem of Landau \cite{Lan29}, closely related to a theorem of Bloch \cite{Blo25}.
\begin{theorem}[Landau, 1929]\label{thm_Landau}
	Let $f: \D \to \C$ be holomorphic from the open unit disc $\D$ to the complex plane $\C$ with the property $|f'(0)| = 1$. Then, there exists an absolute constant $\LC > 0$ such that a disc $D_\LC$ of radius $\LC$ is contained in the image of $f(\D)$.
\end{theorem}
The theorem basically tells us that the unit disc cannot entirely collapse under a mapping with the above properties and that the image must have a diameter of at least $2\LC$.

Landau's problem is to find the exact value of the constant $\LC$ and a precise definition of the constant is as follows. By $\ell(f)$, we denote the radius of the largest disc found in $f(\D)$;
\begin{equation}
	\ell(f) = \sup \{r \in \R_+ \mid D_r \subset f(\D), \, f \textnormal{ as in Theorem \ref{thm_Landau}} \} .
\end{equation}
Landau's constant $\LC$ is then defined as
\begin{equation}
	\LC = \inf \{ \ell(f) \mid f \textnormal{ as in Theorem \ref{thm_Landau} } \},
\end{equation}
which we will also call Landau's ``Weltkonstante", due to Landau's original article \cite{Lan29}.

We note that the problem of finding $\LC$ is invariant under rotation and translation, as for
\begin{equation}
	\widetilde{f}(z) = c f(z) + b, \qquad b, c \in \C, \, |c| = 1,
\end{equation}
we have
\begin{equation}
	| \widetilde{f}'(0) | = |f'(0)|.
\end{equation}
Because of the translation invariance of the problem, often the additional assumption, which is not a restriction, $f(0) = 0$ is made.

We have the following estimates for $\LC$;
\begin{equation}\label{eq_LC}
	\frac{1}{2} + 2 \cdot 10^{-8} < \LC \leq \LC_+ = \frac{\Gamma\left( \tfrac{1}{3} \right) \Gamma\left( \tfrac{5}{6} \right)}{\Gamma\left( \tfrac{1}{6} \right)} \approx 0.543259 \ldots \, .
\end{equation}
The upper bound is conjectured to be sharp and was established by Rademacher in 1943 \cite{Rad43} by constructing a concrete example \footnote{As Rademacher remarked, the same bound was established, but not published, by Robinson already in 1937.}. His example is the universal covering map of the once-punctured (complex) hexagonal torus. Despite some serious effort put into finding the exact value of Landau's constant, see e.g.~the articles \cite{Baernstein_ExtremalProblems_1994}, \cite{Baernstein_HeatKernel_1997}, \cite{Baernstein_Metric_2005}, \cite{BaernsteinVinson_Local_1998}, \cite{Eremenko_Hyperbolic_2011}, the problem remains wide open. The non-strict lower bound $\frac{1}{2}$ was given by Ahlfors \cite{Ahl38} and follows from his theory on ultrahyperbolic metrics (see also \cite{Ahl73}). A seeming improvement that $\LC$ is strictly greater than $\frac{1}{2}$ was achieved by Pommerenke \cite{Pom70}, however, according to the article of Yamada \cite{Yam88} the proof contained a mistake. The best known lower bound was then improved to $\frac{1}{2} +10^{-335}$ by Yanagihara \cite{Yan95} and to $\frac{1}{2} + 2 \cdot 10^{-8}$ by Chen and Shiba \cite{CheShi04}.

The above list of authors, who have contributed to Landau's problem and related Bloch type problems, is of course far from complete and we refer to the references in the above articles.

As described in \cite{BaernsteinVinson_Local_1998}, one can actually focus on universal covering maps of $\C \backslash \Gamma$, where $\Gamma$ is a relatively separated point set in $\C$. Also, the problem can be reformulated by fixing the radius of the disc and asking for the maximal value of the (modulus of the) derivative at the origin. Special cases of relatively separated point sets are lattices, which are discrete, co-compact subgroups of $\C$. A lattice $\L$ can be identified with a two-dimensional torus $\C/\L$. Baernstein suggested to study the heat distribution on the torus to get a better understanding of Landau's problem \cite{Baernstein_ExtremalProblems_1994}, \cite{Baernstein_HeatKernel_1997}, \cite{BaernsteinVinson_Local_1998}. This is the main motivation for this article.

In this work we will solely focus on rectangular tori, i.e., the underlying lattice is rectangular, and hexagonal tori, i.e., the underlying lattice is a hexagonal lattice. We will give a precise description of the coldest and hottest temperature on a rectangular torus of fixed area. Although Landau's problem seems to ask for lattices of fixed covering radius (see e.g.~\cite{Baernstein_HeatKernel_1997} or \cite{Eremenko_Hyperbolic_2011}), we will look at tori of fixed area by means of the (complementary) elliptic modulus. Also, we will try to resolve the seemingly paradox situation of the fixed radius versus fixed area condition. We note that, implicitly, the extremal temperature problem for rectangular tori has been fully treated in the article by Faulhuber and Steinerberger \cite{FaulhuberSteinerberger_Theta_2017} in a manner similar to Montgomery's article \cite{Montgomery_Theta_1988}. The essence in \cite{FaulhuberSteinerberger_Theta_2017} and \cite{Montgomery_Theta_1988} is to find extremal configurations for the periodization of a Gaussian, which results in the study of lattice theta functions. Similar results have also been obtained in the context of lattice energy minimization \cite{Betermin_Cubic_2019}, \cite{BetKnu_Born_18}, \cite{BetKnu18}, \cite{BeterminPetrache_DimensionReduction_2017}. Further topics which are concerned with extremal geometries include the search for extremal determinants of Laplacians on Riemannian surfaces \cite{Osgood_Determinants_1988}, whether one can ``hear the shape of a drum" \cite{Kac66} or the minimization of the Epstein zeta function \cite{Cas59}, \cite{Dia64}, \cite{Enn64}, \cite{Ran53}.

The new aspect in this work is that we can uniquely describe the geometry of a rectangular torus by the ratio of its coldest and hottest point and that this ratio directly refers to the elliptic modulus of the complete elliptic integral of the first kind. By using a remarkable result of Ramanujan, it is possible to determine the coldest and hottest temperature on a rectangular torus by only knowing their ratio. The key behind this fact is, as mentioned, that the elliptic modulus already defines the geometry of the torus. The most stunning fact, however, is that we will show that, for the most natural parameters, the coldest point on the square torus has temperature
\begin{equation}
	G = \theta_4(e^{-\pi})^2 \approx 0.834627 \ldots,
\end{equation}
which is known as Gauss' constant and that the conjectured solution to Eremenko's problem \cite{Eremenko_Hyperbolic_2011} is exactly $2G$. In return, this means that the separable ``Weltkonstante" $\LC_\square$ related to Landau's problem has the conjectured value
\begin{equation}\label{eq_LG}
	\LC_\square = \frac{1}{2G} = \frac{\Gamma\left( \tfrac{1}{2} \right) \Gamma\left( \tfrac{3}{4} \right)}{\Gamma\left( \tfrac{1}{4} \right)} \approx 0.599070 \ldots ,
\end{equation}
which is also known as the second lemniscate constant \cite[Chap.~6]{Fin03}. To the author's knowledge, the connection \eqref{eq_LG} between the lemniscate constant and Landau's problem has not been mentioned before in the literature, even though the value $\LC_\square^{-1} \approx 1.669254 \ldots$ was numerically computed and conjectured to be the solution to the separable Landau problem in \cite{Eremenko_Hyperbolic_2011}. Also, we will show a similar connection of $\LC_+$ to the minimal temperature on the hexagonal torus.

\section{An Extremal Problem for the Heat Kernel on the Torus}\label{sec_hk_rectangular}
In this section we are going to study the temperature distribution on a (real) rectangular torus. This will be done by considering the minimal and the maximal value of the heat kernel associated to the Laplace-Beltrami operator on a torus. For an introduction to heat kernels on manifolds we refer to the textbook of Grigor'yan \cite{Gri_Heat_09}.

We consider the family of lattices
\begin{equation}
	\L_\alpha = \alpha^{-1} \Z \times \alpha \Z, \qquad \alpha \in \R_+,
\end{equation}
of area 1 and the family of resulting tori is given by
\begin{equation}
	\T^2_\alpha = \R^2 \slash \L_\alpha = \R^2 \slash \left(\alpha^{-1} \Z \times \alpha \Z \right).
\end{equation}
We denote the Laplace-Beltrami operator on $\T^2_\alpha$ by $\Delta_\alpha$, where we choose the sign of $\Delta_\alpha$ such that its eigenvalues are non-negative. The eigenfunctions of $\Delta_\alpha$ are the complex exponentials
\begin{equation}
	e^\alpha_{k,l}(x,y) = e^{2 \pi i \left( \alpha k x + \alpha^{-1} l y \right)}
\end{equation}
with $(x,y) \in \T^2_\alpha$ and $(k,l) \in \Z^2$, which means that $(\alpha k, \alpha^{-1} l)$ is an element of the dual lattice $\L_\alpha^\bot = \alpha \Z \times \alpha^{-1} \Z$ . The eigenvalues are given by
\begin{equation}
	\l^\alpha_{k,l} = 4 \pi^2 \left( \alpha^2 k^2 + \alpha^{-2} l^2 \right).
\end{equation}
The associated heat kernel can be written as
\begin{equation}
	p_\alpha((x_1,y_1),(x_2,y_2);t)	:= \sum_{(k,l) \in \Z^2} e^{-\l_{k,l}^\alpha \, t} \, e_{k,l}^\alpha(x_1,y_1) \, \overline{e_{k,l}^\alpha(x_2,y_2)}, \qquad (x_1,y_1), \, (x_2, y_2) \in \T_\alpha^2.
\end{equation}
After making everything explicit, introducing the variable
\begin{equation}
	(x,y) = \left(\alpha (x_1-x_2), \alpha^{-1} (y_1-y_2)\right)
\end{equation}
and scaling $t \mapsto \tfrac{t}{4 \pi}$, this yields
\begin{equation}
	p_\alpha(x,y;t) = \sum_{(k,l) \in \Z^2} e^{- \pi t (\alpha^2 k^2 + \alpha^{-2} l^2)} e^{2 \pi i (k x + l y)},
\end{equation}
which from now on is the heat kernel associated to the Laplace-Beltrami operator $\Delta_\alpha$ on $\T_\alpha^2$. This function is now, of course, periodic with respect to the integer lattice $\Z \times \Z$, and we are actually looking at the torus $\T_1^2$ with a different metric and scaled time (we will give more details in Section \ref{sec_tori}). Nonetheless, by abuse of notation we will still say that we are looking at the heat kernel $p_\alpha(x,y;t)$ associated to the torus $\T_\alpha^2$. We pose the following problem(s), similar in style to Landau's problem. First, we define the minimal and maximal temperature on the torus for a fixed time $t$, given by
\begin{equation}
	A(\alpha;t) = \min_{(x,y)} \, p_\alpha(x,y;t) \qquad \textnormal{ and } \qquad B(\alpha;t) = \max_{(x,y)} \, p_\alpha(x,y;t),
\end{equation}
respectively. We will give an argument for the following claim in the next section, but for the moment, we note that
\begin{equation}
 A(\alpha;t) = p_\alpha \left(\tfrac{1}{2}, \tfrac{1}{2}; t\right) \qquad \textnormal{ and } \qquad B(\alpha;t) = p_\alpha (0,0;t).
\end{equation}
Now, for any $t \in \R_+$ there exist absolute constants $A^*(t)$ and $B_*(t)$ such that
\begin{equation}
	A(\alpha;t) \leq A^*(t) \qquad \textnormal{ and } \qquad B(\alpha;t) \geq B_*(t), \qquad \forall \alpha \in \R_+.
\end{equation}
Moreover, we have
\begin{equation}
	A^*(t) = \sup_{\alpha \in \R_+} A(\alpha;t)  \qquad \textnormal{ and } \qquad B_*(t) = \inf_{\alpha \in \R_+} B(\alpha;t).
\end{equation}
It is most natural, and in fact correct, to assume that
\begin{equation}\label{eq_opt_bounds}
	A^*(t) = A(1,t) \qquad \textnormal{ and } \qquad B_*(t) = B(1,t).
\end{equation}
We will return to the problem and its solution later on, but we already mention that the solution follows from the results in \cite{FaulhuberSteinerberger_Theta_2017}. Also, the problem on finding $B_*$ is closely related to Montgomery's result on minimal theta functions \cite{Montgomery_Theta_1988} and, in fact, the solution follows from the results given in \cite{Montgomery_Theta_1988}. Indeed, by adding the assumption that the quadratic form in Montgomery's theorem is not allowed to have mixed terms, we end up with the result on $B_*(t)$. We note that an extension of the result on $A^*(t)$ to general lattices, analogous to Montgomery's theorem, is still open and that a positive solution, meaning that one can show that the extremizer is the hexagonal lattice, would solve a conjecture of Strohmer and Beaver on optimal lattice configurations for Gaussian Gabor frames \cite{StrBea03}, at least for even density of the lattice (see e.g.~\cite{Faulhuber_Hexagonal_2018}).

\section{Hypergeometric Functions and Theta Functions}\label{sec_special}
As a next step, we will change the topic and study properties of some special functions. The main references for this section are Ramanujan's Notebooks by Berndt, in particular \cite[Chap.~17]{RamanujanIII} and the textbook of Whittaker and Watson \cite[Chap.~21]{WhiWat69}.

For a complex number $z$ and a non-negative integer $k$, we denote the rising Pochhammer symbol by
\begin{equation}
	(z)_k = \frac{\Gamma(z+k)}{\Gamma(z)},
\end{equation}
where $\Gamma(z)$ is Euler's gamma function
\begin{equation}
	\Gamma(z) = \int_{\R_+} t^{z-1} e^{-t} \, dt, \qquad \text{for } Re(z) > 0.
\end{equation}
It extends to a meromorphic function, with poles at the negative integers and 0. The hypergeometric series is then, formally, given by
\begin{equation}
	_mF_n(\alpha_1, \dots, \alpha_m; \, \beta_1, \dots, \beta_n; \, z) = \sum_{k=0}^\infty \frac{(\alpha_1)_k \dots (\alpha_m)_k}{(\beta_1)_k \dots (\beta_n)_k} \frac{z^k}{k!}, \qquad \alpha_1, \ldots, \alpha_m, \beta_1, \ldots, \beta_n,z \in \C
\end{equation}
where $m$ and $n$ are non-negative integers. In this work, we will only consider the case of Gauss' hypergeometric function with real parameters and real variable;
\begin{equation}
	_2 F_1(a,b;c;x) = \sum_{k=0}^\infty \frac{(a)_k (b)_k}{(c)_k} \frac{x^k}{k!},
\end{equation}
and the parameters will usually fulfill $a+b \leq c$ and $x \in (0,1)$, which means that we do not run into convergence issues. A result which we will employ later on is the following formula, due to Gauss (see e.g.~\cite[p.~89, (1.4)]{RamanujanIII}).
\begin{equation}\label{eq_Gauss}
	_2F_1 \left(x,y; \, \tfrac{1}{2}(x+y+1); \, \tfrac{1}{2} \right)
	= \frac{\sqrt{\pi} \, \Gamma \left( \tfrac{1}{2} x + \tfrac{1}{2} y + \tfrac{1}{2} \right)}
		{\Gamma\left( \tfrac{1}{2} x + \tfrac{1}{2} \right) \Gamma\left( \tfrac{1}{2} y + \tfrac{1}{2} \right)}.
\end{equation}

The next family of functions we introduce are Jacobi's theta function, where we will be especially interested in the so-called theta-nulls. We define the theta functions according to the textbook of Whittaker and Watson \cite[Chap.~21]{WhiWat69}. For $z \in \C$ and $q \in \C$ with $|q| < 1$, we define
\begin{align}
	\vartheta_1(z,q) & = \sum_{k \in \Z} (-1)^{(k-1/2)} q^{(k+1/2)^2} e^{(2k+1)\pi i z},
	& &
	\vartheta_2(z,q) = \sum_{k \in \Z} q^{(k+1/2)^2} e^{(2k+1)\pi i z},\\
	\vartheta_3(z,q) & = \sum_{k \in \Z} q^{k^2} e^{2 k \pi i z},
	& &
	\vartheta_4(z,q) = \sum_{k \in \Z} (-1)^k q^{k^2} e^{2 k \pi i z}.
\end{align}
We note that any of the above functions is real-valued for $q \in (0,1)$ as we have Fourier series with real coefficients whose values possess a symmetry in the power $k$. It is also common to write Jacobi's theta functions as functions of the pair of variables $(z,\tau) \in \C \times \mathbb{H}$, where $\mathbb{H}$ is the upper half plane;
\begin{equation}
	\mathbb{H} = \{ z \in \C \mid Im (z) > 0 \}.
\end{equation}
The nome $q$ is then replaced by $e^{\pi i \tau}$ and the fact that $\tau \in \mathbb{H}$ ensures that the series converge. We note that any of the above theta functions is expressible by any other theta function by an appropriate adjustment of the arguments. Also, any of the above theta functions has a product representation, the Jacobi triple product representation for which we refer to the textbook of Whittaker and Watson \cite[Chap.~21]{WhiWat69} or the textbook of Stein and Shakarchi \cite[Chap.~10]{SteSha_Complex_03}. We will only state the product representation for $\vartheta_3(z,q)$. The other product representations can be obtained from this one and, also, we will only need this certain product representation in the sequel;
\begin{equation}
	\vartheta_3(z,q) = \prod_{k \geq 1} \left( 1 - q^{2k} \right) \left( 1 + q^{2k-1} e^{2 \pi i z} \right) \left( 1 + q^{2k-1} e^{-2 \pi i z} \right).
\end{equation}
By expanding
\begin{equation}
	\left( 1 + q^{2k-1} e^{2 \pi i z} \right) \left( 1 + q^{2k-1} e^{-2 \pi i z} \right) = 1 + 2 q^{2k-1} \cos(2 \pi z) + q^{4k-2},
\end{equation}
it readily follows that for any $q \in (0,1)$ we have
\begin{equation}\label{eq_minimum}
	\vartheta_3 \left( \tfrac{1}{2} + l, q \right) \leq \vartheta_3(z,q)
\end{equation}
for any $l \in \Z$, $z \in \R$.

The theta-nulls are functions depending only on $q$ (or $\tau$) and are derived by setting $z = 0$. We note that $\vartheta_1(0,q) = 0$ for all $|q|<1$, as it is an odd function of $z$. The other 3 theta functions are even with respect to $z$ and the theta-nulls are given as follows;
\begin{align}
	\theta_2(q) & = \vartheta_2(0,q) = \sum_{k \in \Z} q^{\left(k + \tfrac{1}{2} \right)^2},\\
	\theta_3(q) & = \vartheta_3(0,q) = \sum_{k \in \Z} q^{k^2},\\
	\theta_4(q) & = \vartheta_4(0,q) = \sum_{k \in \Z} (-1)^k q^{k^2}.
\end{align}
The above functions also obey the following rules (see e.g.~\cite[Chap.~4]{ConSlo99} or \cite[Chap.~21]{WhiWat69}), which can be established by using the Poisson summation formula;
\begin{equation}\label{eq_Poisson3}
	\theta_3 \left(e^{\pi i \tau} \right) = \sqrt{\tfrac{i}{\tau}} \, \theta_3 \left(e^{-\pi \tfrac{i}{\tau}} \right),
\end{equation}
\begin{equation}\label{eq_Poisson24}
	\theta_2 \left(e^{\pi i \tau} \right) = \sqrt{\tfrac{i}{\tau}} \, \theta_4 \left(e^{-\pi \tfrac{i}{\tau}} \right)
	\qquad  \textnormal{ and } \qquad
	\theta_4 \left(e^{\pi i \tau} \right) = \sqrt{\tfrac{i}{\tau}} \, \theta_2 \left(e^{-\pi \tfrac{i}{\tau}} \right).
\end{equation}

There exist most beautiful connections between the theta-nulls of Jacobi's theta functions and Gauss hypergeometric functions, the first of which we find in Berndt's Part III of Ramanujan's Notebooks \cite[Chap.~17, Entry 6]{RamanujanIII}.
\begin{equation}\label{eq_magic}
	{}_2F_1 \left(\tfrac{1}{2}, \tfrac{1}{2}; 1; k^2\right) = \theta_3(q)^2,
\end{equation}
where the quantity $k$ is called the elliptic modulus or eccentricity. It is defined as
\begin{equation}
	k = \frac{\theta_2(q)^2}{\theta_3(q)^2}
\end{equation}
and, hence, depends implicitly on $q = e^{\pi i \tau}$, $\tau \in \mathbb{H}$. In the sequel, we will also encounter the quantity $k'$, called the complementary elliptic modulus \footnote{Here and in the rest of the work, the expression $k'$ should not be mistaken for the derivative of $k$.};
\begin{equation}
	k^2 + k'^2 = 1.
\end{equation}
The name elliptic modulus actually comes from its appearance in the complete elliptic integrals of the first kind. The connection is as follows (see e.g.~\cite[p.~4, (I6)]{RamanujanIII});
\begin{equation}
	K(k) = \int_0^{\pi/2} \frac{d\varphi}{\sqrt{1-k^2 \sin(\varphi)^2}} = \frac{\pi}{2} \; {}_2F_1\left( \tfrac{1}{2}, \tfrac{1}{2}; 1; k^2 \right).
\end{equation}
The reader interested in this and other relations may also consult \cite{RamanujanIII}, \cite{Cox_Mean_1984} or \cite[Chap.~21-22]{WhiWat69}.

It is well known (see e.g.~\cite[Chap.~4]{ConSlo99}) that
\begin{equation}\label{eq_4}
	\theta_2(q)^4 + \theta_4(q)^4 = \theta_3(q)^4 .
\end{equation}
Hence, it follows that
\begin{equation}
	k' = \frac{\theta_4(q)^2}{\theta_3(q)^2}.
\end{equation}
We note that elliptic integrals are sometimes also parametrized by the so-called parameter $m$ instead of the modulus $k$ (see e.g.~\cite[Chap.~16-17]{AbrSte72}). The connection is rather simple;
\begin{equation}
	m = k^2 = \frac{\theta_2(q)^4}{\theta_3(q)^4}.
\end{equation}
The complementary parameter $m'$ is connected to the parameter by
\begin{equation}
	m + m' = 1.
\end{equation}
Consequently, we also get
\begin{equation}
	m' = k'^2 = \frac{\theta_4(q)^4}{\theta_3(q)^4}.
\end{equation}
We will state our results in terms of the modulus and complementary modulus $k$ and $k'$. However, the results can easily be formulated in terms of the parameter and the complementary parameter as well.

In the sequel, it will be convenient to introduce the following real-valued theta functions of one non-negative (real) argument;
\begin{equation}
	\widetilde{\theta}_j(t) := \theta_j(e^{-\pi t}), \qquad j = 2,3,4.
\end{equation}
The above functions are hence the theta-nulls of Jacobi's theta functions restricted to $q = e^{-\pi t}$, or $\tau = i t$ with $t \in \R_+$. We note the following property.
\begin{proposition}\label{pro_k}
	The map
	\begin{align}
		\widetilde{k}: \R_+ & \to (0,1)\\
		t & \mapsto \widetilde{k}(t) = \frac{\widetilde{\theta}_2(t)^2}{\widetilde{\theta}_3(t)^2}
	\end{align}
	is bijective and strictly decreasing for $t$ increasing.
\end{proposition}
\begin{proof}
	We will first show that $\widetilde{k}(t)$ is a decreasing function of $t$. We start by differentiating $\widetilde{k}$ with respect to $t$;
	\begin{equation}
		\dfrac{d}{dt} \widetilde{k}(t) = 2 \frac{\widetilde{\theta}_2(t)}{\widetilde{\theta}_3(t)^3} \left( \widetilde{\theta}_2'(t) \widetilde{\theta}_3(t) - \widetilde{\theta}_2(t) \widetilde{\theta}_3'(t)\right).
	\end{equation}
	Since $\widetilde{\theta}_2$ and $\widetilde{\theta}_3$ are positive, in order to show that $\widetilde{k}(t)$ is strictly decreasing, it suffices to show that
	\begin{equation}
		\widetilde{\theta}_2'(t) \widetilde{\theta}_3(t) < \widetilde{\theta}_3'(t) \widetilde{\theta}_2(t),
	\end{equation}
	which is equivalent to showing
	\begin{equation}
		t \frac{\widetilde{\theta}_2'(t)}{\widetilde{\theta}_2(t)} < t \frac{\widetilde{\theta}_3'(t)}{\widetilde{\theta}_3(t)}.
	\end{equation}
	We set $\phi_j (t) = t \frac{\widetilde{\theta}_j'(t)}{\widetilde{\theta}_j(t)}$ for $j \in \{2,3\}$ and use the following results given in \cite{Faulhuber_Determinants_2018} (see also \cite{FaulhuberSteinerberger_Theta_2017}). The function $\phi_2$ is strictly decreasing whereas the function $\phi_3$ is strictly increasing. Also
	\begin{equation}	
		\lim_{t \to 0} \phi_2(t) = \lim_{t \to 0} \phi_3(t) = - \frac{1}{2}.
	\end{equation}
	This proves the strict monotonicity. To show that the map is bijective between $\R_+$ and the interval $(0,1)$, we note that
	\begin{equation}
		\lim_{t \to 0} \frac{\widetilde{\theta}_2(t)}{\widetilde{\theta}_3(t)} = 1 \qquad \textnormal{ and } \qquad \lim_{t \to \infty} \frac{\widetilde{\theta}_2(t)}{\widetilde{\theta}_3(t)} = 0.
	\end{equation}
\end{proof}
Proposition \ref{pro_k} hence describes the behavior of the elliptic modulus $k(e^{-\pi t})$ (or the parameter $m(e^{-\pi t})$) as a function of $t$. The behavior is illustrated in Figure \ref{fig_modulus}.
\begin{figure}[htb]
	\subfigure[The behavior of $k(e^{-\pi t})$ as a function of $t$. For $t = 1$ the value of $k$ is $\tfrac{1}{\sqrt{2}}$.]
	{
		\includegraphics[width=.45\textwidth]{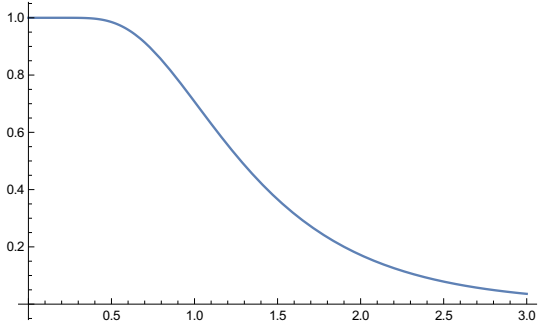}
	}
	\hfill
	\subfigure[The behavior of $m(e^{-\pi t})$ as a function of $t$. For $t = 1$ the value of $m$ is $\tfrac{1}{2}$.]
	{
		\includegraphics[width=.45\textwidth]{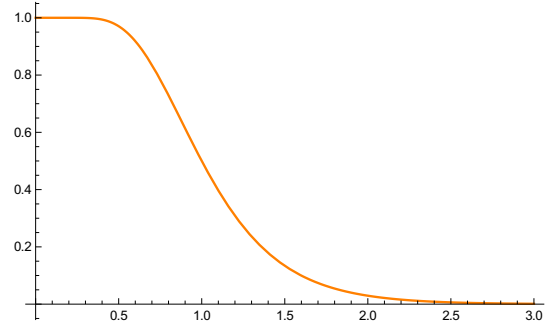}
	}
	\caption{The behavior of the elliptic modulus $k(e^{-\pi t})$ and the parameter $m(e^{-\pi t})$ as functions of the real variable $t$ as described in Proposition \ref{pro_k}.}\label{fig_modulus}
\end{figure}

We note that the result in Proposition \ref{pro_k} was also established in \cite[Prop.~3.7.]{BeterminPetrache_DimensionReduction_2017}, by showing, with more direct methods, that the complementary elliptic modulus is strictly increasing.

\section{Consequences for the Heat Kernel}\label{sec_consequences_hk}
We will now use the collected results to describe the behavior of the temperature on a rectangular torus. We start with the observation that the heat kernel can actually be written as a product of theta functions. For any $\alpha \in \R_+$, we have
\begin{equation}
	p_\alpha(x,y;t) = \vartheta_3 \left(x,e^{-\pi t \alpha^2} \right) \, \vartheta_3 \left(y,e^{-\pi t \alpha^{-2}} \right), \qquad x,y \in \R, \, t \in \R_+ \; .
\end{equation}
We will now give the arguments for the facts that
\begin{equation}\label{eq_min_max}
	A(\alpha;t) = p_\alpha \left( \tfrac{1}{2}, \tfrac{1}{2}; t\right) \qquad \textnormal{ and } \qquad B(\alpha;t) = p_\alpha \left( 0,0; t\right).
\end{equation}
The equation for $A$ follows from the product representation of $\vartheta_3$ and equation \eqref{eq_minimum}. Actually, from \eqref{eq_minimum} we get that
\begin{equation}
	p_\alpha \left( \tfrac{1}{2} + l_1, \tfrac{1}{2} + l_2; t\right) \leq p_\alpha(x,y;t), \qquad l_1, l_2 \in \Z, \, x,y \in \R,
\end{equation}
but since $p_\alpha$ is periodic with period $\Z^2$, we may focus on solutions in $[0,1) \times [0,1)$. The equation for $B$ follows readily by using the triangle inequality;
\begin{equation}
	p_\alpha(x,y;t) \leq \sum_{(k,l) \in \Z^2} \left|e^{- \pi t (\alpha^2 k^2 + \alpha^{-2} l^2)} \right| \left|e^{2 \pi i (k x + l y)} \right| = p_\alpha(m,n;t), \qquad (m,n) \in \Z^2.
\end{equation}
The results in \eqref{eq_min_max} are special cases of \cite[Prop.~3.7.]{BetKnu_Born_18} and \cite[Prop.~3.4.]{BeterminPetrache_DimensionReduction_2017}, respectively and were also derived in \cite[Sec.~6]{Jan96}.

We can now write $A$ and $B$ in terms of theta-nulls;
\begin{equation}
	A(\alpha;t) = \theta_4 \left(e^{-\pi t \alpha^2} \right) \, \theta_4 \left(e^{-\pi t \alpha^{-2}} \right)
	\qquad \textnormal{ and } \qquad
	B(\alpha;t) = \theta_3 \left(e^{-\pi t \alpha^2} \right) \, \theta_3 \left(e^{-\pi t \alpha^{-2}} \right).
\end{equation}
In \cite{FaulhuberSteinerberger_Theta_2017} we find the following result, which we adapted to the notation we use in this article.
For $\alpha, t \in \R_+$, the following is true;
\begin{equation}\label{eq_FauSteA}
	\theta_4 \left(e^{-\pi t \alpha^2} \right) \, \theta_4 \left(e^{-\pi t \alpha^{-2}} \right) \leq \theta_4 \left(e^{-\pi t} \right) \, \theta_4 \left(e^{-\pi t} \right).
\end{equation}
\begin{equation}\label{eq_FauSteB}
	\theta_3 \left(e^{-\pi t \alpha^2} \right) \, \theta_3 \left(e^{-\pi t \alpha^{-2}} \right) \geq \theta_3 \left(e^{-\pi t} \right) \, \theta_3 \left(e^{-\pi t} \right).
\end{equation}
In both cases equality holds if and only if $\alpha = 1$.

This shows that, for any fixed time $t$, among all rectangular tori the square torus uniquely maximizes the lowest temperature and uniquely minimizes the highest temperature. Recently, it was shown in \cite{Faulhuber_Curious_2019} that, for $\alpha \in \R_+$,
\begin{align}
	& \theta_3 \left(e^{-\pi t \alpha^2} \right) \, \theta_3 \left(e^{-\pi t \alpha^{-2}} \right) \geq \theta_3 \left(e^{-\pi t} \right) \, \theta_3 \left(e^{-\pi t} \right),\qquad \forall t \in \R_+\\
	\Longrightarrow & \; \theta_4 \left(e^{-\pi \alpha^2} \right) \, \theta_4 \left(e^{-\pi \alpha^{-2}} \right) \leq \theta_4 \left(e^{-\pi} \right) \, \theta_4 \left(e^{-\pi} \right).
\end{align}
Again, equality holds if and only if $\alpha = 1$.

We note that the above result needs the information that the square torus minimizes the highest temperature for all times to derive the analogous statement for the lowest temperature for $t=1$ as an implication. So far, the author was not able to extend the result to tori associated to arbitrary lattices. In particular, it would be very interesting to know whether Montgomery's result \cite{Montgomery_Theta_1988} already implies that the hexagonal torus uniquely maximizes the lowest temperature (at least for $t=1$) among all (regular) tori.

As a next step, we will give a different interpretation to the elliptic modulus and the parameter $t$. The usual interpretation of $t$ being time is one possibility, another possibility is to see it as the density of the lattice (which is the reciprocal of the area of the lattice), or, we could also say it is the product of time and density. We will have a look at what happens if we say time is fixed to 1 and $t$ represents the density of the rectangular lattice. In this case, the rectangular torus is represented by
\begin{equation}
	\T^2_{(\alpha;t)} = \R^2 \Big\slash \left( \tfrac{1}{\sqrt{t}} \left(\alpha^{-1} \Z \times \alpha \Z \right) \right) = \R^2 \slash \L_{(\alpha;t)}
\end{equation}
and its surface area is $t^{-1}$. By abusing notation once more, the associated heat kernel is (still) $p_{(\alpha;t)}(x,y;1) = p_\alpha (x,y;t)$. Furthermore, we now force our tori to be square, i.e., we set $\alpha = 1$. This leads to the following result.
\begin{lemma}\label{lem_temp}
	Let $p_{(1;t)}$ be the heat kernel of the square torus of surface area $t^{-1}$, $t \in \R_+$ and let $k' \in (0,1)$ be the ratio of the coldest and hottest temperature. Then, the coldest and hottest temperature on the torus $\T_{(1;t)}^2$ are given by
	\begin{equation}
		A(1;t) = k' \, {}_2F_1\left( \tfrac{1}{2}, \tfrac{1}{2}; 1; 1-k'^2 \right)
		\qquad \textnormal{ and } \qquad
		B(1;t) = {}_2F_1\left( \tfrac{1}{2}, \tfrac{1}{2}; 1; 1-k'^2 \right)
	\end{equation}
	respectively.
\end{lemma}
\begin{proof}
	By using the connection in \eqref{eq_4} we find out that
	\begin{equation}\label{eq_modulus_tempreature}
		k(e^{-\pi t})^2 = \frac{\theta_2(e^{-\pi t})^4}{\theta_3(e^{-\pi t})^4} = 1 - \frac{\theta_4(e^{-\pi t})^4}{\theta_3(e^{-\pi t})^4} = 1 - \frac{A(1;t)^2}{B(1;t)^2}.
	\end{equation}
	Hence, the complementary elliptic modulus precisely describes the behavior of the ratio of the coldest and warmest point on the torus as time evolves (linearly);
	\begin{equation}
		k'(e^{-\pi t})^2 = \frac{\theta_4(e^{-\pi t})^4}{\theta_3(e^{-\pi t})^4} = \frac{A(1;t)^2}{B(1;t)^2}.
	\end{equation}
	By applying the, now seemingly magical, formula \eqref{eq_magic} of Ramanujan we get for a given value $k = k(e^{-\pi t}) \in (0,1)$ that
	\begin{equation}
		{_2F_1} \left(\tfrac{1}{2}, \tfrac{1}{2}; 1; k^2\right) = \theta_3(e^{-\pi t})^2 = B(1;t).
	\end{equation}
	It readily follows that
	\begin{equation}
		k' \, _2F_1 \left(\tfrac{1}{2}, \tfrac{1}{2}; 1; k^2\right) = k' \, \theta_3(e^{-\pi t})^2 = \theta_4(e^{-\pi t})^2 = A(1;t).
	\end{equation}
\end{proof}
This reveals a truly remarkable aspect of Ramanujan's formula \eqref{eq_magic}, as, for any time (or density) $t \in \R_+$, it allowed Ramanujan to split the complementary elliptic modulus $k'$ into the hottest and coldest temperature on the square torus $\T_{(1;t)}^2$ by only knowing their ratio. This leads to the following result for rectangular tori of surface area 1 and fixed time equal to 1.
\begin{theorem}\label{thm_T4}
	Let time be fixed to 1, set $\L_{(\alpha;1)} = \alpha^{-1} \Z \times \alpha \Z$ and consider the heat kernel $p_{(\alpha;1)}(x,y;1)$ on the torus $\T_{(\alpha;1)}^2 = \R^2 \slash \L_{(\alpha;1)}$. Let $k = k(e^{-\pi \alpha^2})$ and $k' = k'(e^{-\pi \alpha^2})$ be the elliptic and complementary elliptic modulus. Then, the minimal and maximal temperature are given by
	\begin{align}
		A(\alpha;1) & = A(k') = \sqrt{k \,k' \, {_2F_1} \left(\tfrac{1}{2}, \tfrac{1}{2}; 1; k^2\right) {_2F_1} \left(\tfrac{1}{2}, \tfrac{1}{2}; 1; k'^2\right)},\\
		B(\alpha;1) & = B(k') = \sqrt{{_2F_1} \left(\tfrac{1}{2}, \tfrac{1}{2}; 1; k^2\right) {_2F_1} \left(\tfrac{1}{2}, \tfrac{1}{2}; 1; k'^2\right)},
	\end{align}
	respectively.
\end{theorem}
\begin{proof}
	As a consequence of Proposition \ref{pro_k} we know that any $\alpha \in \R_+$ can be uniquely identified with a value of the complementary elliptic modulus $k' \in (0,1)$.
	
	The next step is to study the temperature distribution on a certain family of 4-dimensional tori. This step might seem artificial at first, but our intentions will become clear rather soon. We consider the following family of tori;
	\begin{equation}
		\T^4_{(\alpha, \alpha^{-1})} = \R^4 \big/ \left(\alpha^{-1} \Z \times \alpha^{-1} \Z \times \alpha \Z \times \alpha \Z\right) = \T^2_{(1;\alpha^2)} \times \T^2_{(1;\alpha^{-2})}.
	\end{equation}
	Now, the heat kernel on this new, 4-dimensional torus is just the tensor product of the two heat kernels on the two square tori of dimension 2 with different densities, i.e.,
	\begin{equation}
		p_{(1;\alpha^2)} (x_1,y_1;1) \, p_{(1;\alpha^{-2})}(x_2,y_2;1), \qquad (x_1,y_1,x_2,y_2) \in \T^4_{(\alpha,\alpha^{-1})}.
	\end{equation}		
	However, after re-labeling this is just the tensor product of twice the same rectangular torus of density 1, i.e.,
	\begin{align}
		p_{(1;\alpha^2)} (x_1,y_1;1) \, p_{(1;\alpha^{-2})}(x_2,y_2;1)
		& = p_{(\alpha,1)} (x_1,y_1;1) \, p_{(\alpha^{-1},1)}(x_2,y_2;1)\\
		& = \vartheta_3(x_1,e^{-\pi \alpha^2}) \vartheta_3(y_1,e^{-\pi \alpha^2}) \vartheta_3(x_2,e^{-\pi \alpha^{-2}}) \vartheta_3(y_2,e^{-\pi \alpha^{-2}}).
	\end{align}		
	By the definition of $k(q)$ and $k'(q)$ and equations \eqref{eq_Poisson3} and \eqref{eq_Poisson24}, we get
	\begin{equation}
		k'(e^{-\pi \alpha^2}) = k(e^{-\pi \alpha^{-2}}).
	\end{equation}
	Now, by Lemma \ref{lem_temp} we conclude that the hottest and coldest point on $\T^4_{(\alpha, \alpha^{-1})}$ have temperature
	\begin{equation}
		{_2F_1} \left(\tfrac{1}{2}, \tfrac{1}{2}; 1; k^2\right) {_2F_1} \left(\tfrac{1}{2}, \tfrac{1}{2}; 1; k'^2\right) = \theta_3(e^{-\pi \alpha^2})^2 \theta_3(e^{-\pi \alpha^{-2}})^2
	\end{equation}
	and
	\begin{align}
		k \,k' \, {_2F_1} \left(\tfrac{1}{2}, \tfrac{1}{2}; 1; k^2\right) {_2F_1} \left(\tfrac{1}{2}, \tfrac{1}{2}; 1; k'^2\right)
		& = \theta_2(e^{-\pi \alpha^2})^2 \theta_2(e^{-\pi \alpha^{-2}})^2\\
		& = \theta_4(e^{-\pi \alpha^2})^2 \theta_4(e^{-\pi \alpha^{-2}})^2
	\end{align}
	respectively. By construction, it follows that for the torus $\T^2_{(\alpha;1)}$, we have the desired results;
	\begin{align}
		A(\alpha;1) = A(k') = \sqrt{k \,k' \, {_2F_1} \left(\tfrac{1}{2}, \tfrac{1}{2}; 1; k^2\right) {_2F_1} \left(\tfrac{1}{2}, \tfrac{1}{2}; 1; k'^2\right)}
		& = \theta_2(e^{-\pi \alpha^2}) \theta_2(e^{-\pi \alpha^{-2}})\\
		& = \theta_4(e^{-\pi \alpha^2}) \theta_4(e^{-\pi \alpha^{-2}}).
	\end{align}
	\begin{equation}
		B(\alpha;1) = B(k') = \sqrt{ {_2F_1} \left(\tfrac{1}{2}, \tfrac{1}{2}; 1; k^2\right) {_2F_1} \left(\tfrac{1}{2}, \tfrac{1}{2}; 1; k'^2\right)} = \theta_3(e^{-\pi \alpha^2}) \theta_3(e^{-\pi \alpha^{-2}}).
	\end{equation}
\end{proof}
By combining Theorem \ref{thm_T4} with formulas \eqref{eq_FauSteA} and \eqref{eq_FauSteB} we get the following result.
\begin{corollary}
	It follows that
	\begin{equation}
		A(k') \leq A \left(\tfrac{1}{\sqrt{2}}\right) \textnormal{ and } B(k') \geq B \left(\tfrac{1}{\sqrt{2}}\right)
	\end{equation}
	for all $k' \in (0,1)$, with equality if and only if $k' = k = \frac{1}{\sqrt{2}}$.
\end{corollary}
The case $k' = k = \tfrac{1}{\sqrt{2}}$ corresponds to the case of the square torus, i.e., $\alpha = 1$. That is best seen by using \eqref{eq_Poisson3} and \eqref{eq_Poisson24} which yield, for $q = e^{-\pi}$ ($\Leftrightarrow \alpha = 1 \Leftrightarrow \tau = i$),
\begin{equation}
	k(e^{-\pi})^2 = k'(e^{-\pi})^2 = \frac{\theta_2(e^{-\pi})^4}{\theta_3(e^{-\pi})^4} = \frac{\theta_4(e^{-\pi})^4}{\theta_3(e^{-\pi})^4} = \frac{A(1;1)^2}{B(1;1)^2} = \frac{1}{2}.
\end{equation}
So, among all rectangular tori, the lowest temperature is maximal for the square torus; likewise the hottest temperature is minimal for the square torus. In the following figures we have the temperatures and their ratio plotted as a function of the complementary elliptic modulus and the complementary parameter. We note that the plots are invariant under the substitution $m' \mapsto 1-m'$. This symmetry is not apparent when we use the modulus.

\begin{figure}[ht]
	\subfigure[Behavior of the minimal temperature $A(k')$ as a function of the complementary elliptic modulus.]
	{
		\includegraphics[width=.45\textwidth]{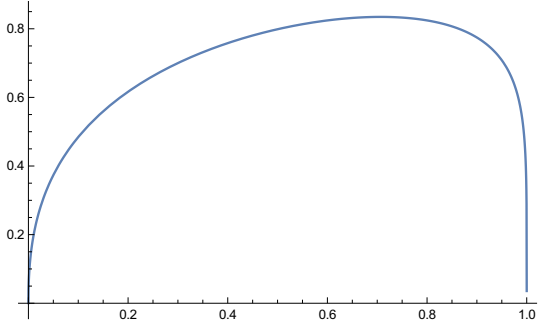}
	}
	\hfill
	\subfigure[Behavior of the minimal temperature $A(m')$ as a function of the complementary parameter.]
	{
		\includegraphics[width=.45\textwidth]{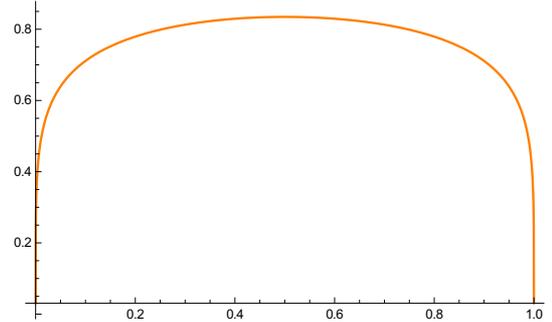}
	}
	\caption{The minimal temperature is maximal if $k'^2 = m' = \tfrac{1}{2}$ and the value is $\tfrac{1}{\sqrt{2}} \, {_2F_1} \left(\tfrac{1}{2}, \tfrac{1}{2}; 1; \tfrac{1}{2}\right) \approx 0.834627 \ldots$ .}
\end{figure}
\begin{figure}[ht]
	\subfigure[Behavior of the maximal temperature $B(k')$ as a function of the complementary elliptic modulus.]
	{
		\includegraphics[width=.45\textwidth]{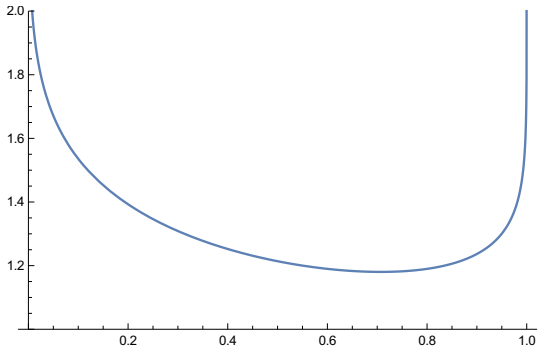}
	}
	\hfill
	\subfigure[Behavior of the maximal temperature $B(m')$ as a function of the complementary parameter.]
	{
		\includegraphics[width=.45\textwidth]{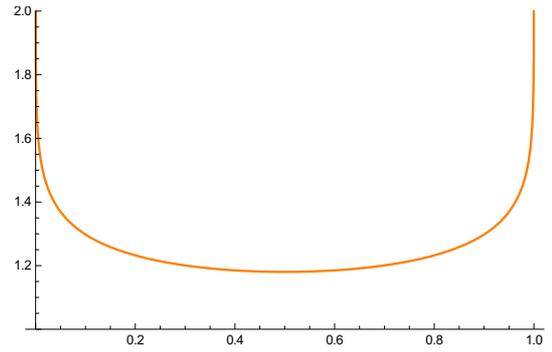}
	}
	\caption{The maximal temperature is minimal if $k'^2 = m' = \tfrac{1}{2}$ and the value is ${_2F_1} \left(\tfrac{1}{2}, \tfrac{1}{2}; 1; \tfrac{1}{2}\right) \approx 1.18034 \ldots$ .}
\end{figure}
\begin{figure}[ht]
	\subfigure[Behavior of the temperature ratio $A(k')/B(k')$ as a function of the complementary elliptic modulus.]
	{
		\includegraphics[width=.45\textwidth]{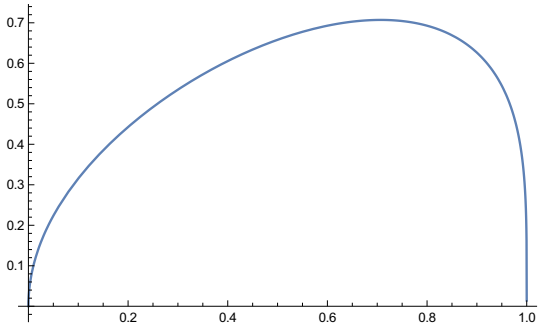}
	}
	\hfill
	\subfigure[Behavior of the temperature ratio $A(m')/B(m')$ as a function of the complementary parameter.]
	{
		\includegraphics[width=.45\textwidth]{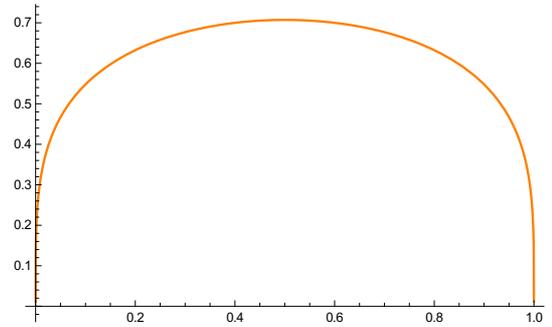}
	}
	\caption{The temperature ratio $A/B$ is most balanced if $k'^2 = m' = \tfrac{1}{2}$ and the value is $\tfrac{1}{\sqrt{2}} \approx 0.707107 \ldots$ .}
\end{figure}

\section{The Heat Kernel Problem for General Tori}\label{sec_tori}
We will now give a brief description of the heat kernel on general tori
\begin{equation}
	\T^2_\L = \R^2 \slash \L,
\end{equation}
where $\L$ is a lattice in $\R^2$ with area 1. A lattice of area 1 is a discrete, co-compact subgroup of $\R^2$ and can be represented by a matrix $M \in SL(2,\R)$;
\begin{equation}
	\L = M \Z^2.
\end{equation}
The columns $v_1$, $v_2$ of the matrix $M = (v_1,v_2)$ serve as the basis for $\L$, which is the integer span of $v_1$ and $v_2$;
\begin{equation}
	\L = \{k v_1 + l v_2 \mid (k,l) \in \Z^2, M = (v_1, v_2) \in SL(2,\R)\}.
\end{equation}
The matrix generating a lattice is not unique, as we can choose another basis for the lattice. In particular, if $\mathfrak{B} \in SL(2,\Z)$ (it only has integer entries and determinant 1), then
\begin{equation}
	\L = (M \mathfrak{B}) \Z^2 = M \Z^2.
\end{equation}
For details on the group $SL(2,\Z)$ and its properties we refer to the textbook by Serre \cite{Serre_1973}.

First, we recall the special case of the standard torus $\T^2 = \R^2 \slash \Z^2$. Its heat kernel is given by
\begin{align}
	p_{\Z^2}(x,y;t) & = \sum_{(k,l) \in \Z^2} e^{- \pi t (k^2 + l^2)} e^{2 \pi i (k x + l y)}\\
	& = \tfrac{1}{t} \sum_{(k,l) \in \Z^2} e^{- \pi \tfrac{1}{t} \left((k+x)^2 + (l+y)^2\right)}, \quad (x,y) \in \T^2, \, t \in \R_+,
\end{align}
where the second equality follows from the Poisson summation formula.

In order to gain a more complete picture, we introduce, in the spirit of Montgomery, the following (real) lattice theta function;
\begin{equation}
	\widetilde{p}_{\L}(z;t) = \sum_{\l \in \L} e^{- \pi t (\l + z)^2}, \quad z \in \R^2, \, t \in \R_+.
\end{equation}
By using the Poisson summation formula, it is possible to express the heat kernel on $\T^2_\L$ by this lattice theta function;
\begin{equation}\label{eq_lattice_theta}
	p_{\L^\perp}(z;t) = \sum_{\l^\perp \in \L^\perp} e^{- \pi t {\l^\perp}^2} e^{2 \pi i \l^\perp \cdot z}
	= \tfrac{1}{t} \sum_{\l \in \L} e^{- \pi \tfrac{1}{t} (\l + z)^2} = \tfrac{1}{t} \, \widetilde{p}_{\L}(z;\tfrac{1}{t}),
\end{equation}
where $\l^\perp \cdot z$ is the Euclidean inner product of $\l^\perp$ and $z$, ${\l^\perp}^2 = \l^\perp \cdot \l^\perp$ and
\begin{equation}
	\L^\perp = M^{-T} \Z^2 = \left(M^{-1}\right)^T \Z^2
\end{equation}
is the dual lattice to $\L$. Note that both functions in \eqref{eq_lattice_theta} are $\L$-periodic and express the heat kernel $p_\L(z;t)$. However, the notation in \eqref{eq_lattice_theta} seems more appropriate and is more consistent with the concepts and notation we are about to introduce below.

For lattices in $\R^2$ of unit area, the dual lattice is just a 90 degrees rotated version of the original lattice. This is sometimes irritating when working with quadratic forms associated to lattices and can be overcome by using the symplectic version of the Poisson summation formula. A matrix $S \in SL(2d, \R)$, $d \in \N$, is called symplectic, if and only if
\begin{equation}\label{eq_symplectic}
	S J S^T = J,
\end{equation}
where
$
	J = \begin{pmatrix}
		0 & I\\
		-I & 0
	\end{pmatrix}
$ is the standard symplectic matrix and $I$ is the identity matrix in $SL(d,\R)$. We note that symplectic matrices form a group under matrix multiplication, denoted by $Sp(d)$, which, in general is a subset of $SL(2d,\R)$.

We note that any $S \in SL(2,\R)$ is symplectic, i.e., $Sp(1) = SL(2,\R)$. Let $\L = S \Z^2$, $S \in SL(2, \R)$, then the adjoint lattice $\L^\circ$ is given by
\begin{equation}\label{eq_sympletic_lattice_adjoint}
	\L^\circ = J S^{-T} \Z^2 = \underbrace{J S^{-T} J^{-1}}_{= S} \Z^2 = \L,
\end{equation}
as $J^{-1} \Z^2 = \Z^2$ and \eqref{eq_symplectic}. The standard symplectic form is given by
\begin{equation}
	\sigma(z,z') = z \cdot J z' = x y' - x' y, \quad z=(x,y), \, z'=(x',y')
\end{equation}
and replaces the Euclidean inner product in symplectic geometry and is also used in Hamiltonian mechanics \cite{Gos11}. The symplectic Fourier transform is given by
\begin{equation}
	\mathcal{F}_\sigma f(z) = \int_{\R^2} f(z) e^{- 2 \pi i \, \sigma(z,z')} \, dz', \quad z \in \R^2,
\end{equation}
for suitable $f$, e.g., in the Schwartz space. With these definitions at hand, we can now introduce the symplectic Poisson summation formula as used in \cite{Faulhuber_Note};
\begin{equation}
	\sum_{\l \in \L} f(\l + z) = \sum_{\l^\circ \in \L^\circ} \mathcal{F}_\sigma f(\l^\circ) \, e^{2 \pi i \, \sigma(\l^\circ, z)}.
\end{equation}
For more details on the symplectic Fourier transform we refer to \cite{Gos11}.

With this tool, we have yet another way to express the heat kernel on the torus $\T^2_\L$;
\begin{equation}\label{eq_symplectic_theta}
	p_{\L^\circ}(z;t) = \sum_{\l^\circ \in \L^\circ} e^{- \pi t {\l^\circ}^2} e^{2 \pi i \, \sigma(\l^\circ, z)}
	= \tfrac{1}{t} \sum_{\l \in \L} e^{- \pi \tfrac{1}{t} (\l + z)^2} = \tfrac{1}{t} \, \widetilde{p}_{\L}(z;\tfrac{1}{t}) 
\end{equation}
Since we are dealing with symplectic lattices, i.e., $\L = S \Z^2$ with $S \in Sp(1) = SL(2, \R)$, we can replace $\L^\circ$ by $\L$ due to \eqref{eq_sympletic_lattice_adjoint}. Note that the difference between \eqref{eq_lattice_theta} and \eqref{eq_symplectic_theta} is actually only given by a rotation of the dual lattice by 90 degrees, making it the adjoint lattice, and using the symplectic form instead of the Euclidean inner product in the Fourier series to compensate for the rotation. However, in 2 dimensions any lattice of unit area is its own adjoint lattice and therefore we have the following Jacobi-like identity;
\begin{equation}
	p_\L(z;t) = \tfrac{1}{t} \, \widetilde{p}_\L(z;\tfrac{1}{t}).
\end{equation}
In particular, for $z = 0$ this yields
\begin{equation}
	p_\L(0;t) = \tfrac{1}{t} \, p_\L(0;\tfrac{1}{t}).
\end{equation}
Furthermore, we note that the functions involved in the symplectic Poisson summation formula are 2-dimensional, normalized Gaussians and they are eigenfunctions of the symplectic Fourier transform with eigenvalue 1. This is an easy adaption of the result on the Fourier transform of Gaussians in Folland's textbook \cite[Append.~A]{Fol}, which involves the dual lattice.

Alternatively, the heat kernel on the torus $\T^2_\L$ can also be seen as the heat kernel on the standard torus $\T^2$ with a different metric. Staying close to the notation from above we get;
\begin{align}\label{eq_hk_metric}
	\widetilde{p}_{(\Z^2, S)}(z;\tfrac{1}{t}) & = \tfrac{1}{t} \sum_{\l \in \Z^2} e^{- \pi \tfrac{1}{t} \left( S \left(\l + z\right) \right)^2}
	= \tfrac{1}{t} \sum_{\l \in \Z^2} e^{- \pi \tfrac{1}{t} \left(\l + z\right)^T S^T S \left(\l + z \right)}\\
	& = \sum_{\l \in \Z^2} e^{- \pi t \, \l^T S^T S \l} e^{2 \pi i \, \sigma(\l,z)} = p_{(\Z^2,S)}(z;t).
\end{align}
where $z \in \T^2$ and $S$ is a symplectic matrix defining the metric \footnote{Note that in \eqref{eq_hk_metric} the sum is over the integer lattice $\Z^2$.}. The equality to the second line in \eqref{eq_hk_metric} is the result of using the symplectic Poisson summation formula and the fact that 2-dimensional Gaussians are eigenfunctions of the symplectic Fourier transform with eigenvalue 1 (see, e.g.~\cite{Faulhuber_Note}).

Similar to Section \ref{sec_hk_rectangular}, we denote the minimal and the maximal temperature on a general torus of unit area by
\begin{equation}
	A_\L(t) = \min_{z \in \T^2_{\L}} p_{\L}(z;t)
	\qquad \textnormal{ and } \qquad
	B_\L(t) = \max_{z \in \T^2_{\L}} p_{\L}(z;t),
\end{equation}
respectively. We define the following constants, similar in style to Landau's ``Weltkonstante";
\begin{equation}\label{eq_AB}
	\mathcal{A}^*(t) = \max_{\L} A_\L(t)
	\qquad \textnormal{ and } \qquad
	\mathcal{B}_*(t) = \min_{\L} B_\L(t).
\end{equation}
The corresponding problem is, for fixed $t$, to find the exact values of $\mathcal{A}^*$ and $\mathcal{B}_*$. As we have seen, if we restrict our attention to rectangular tori, then the answer for both cases is derived for the square torus. We will deal with the corresponding problem for general tori in the next section.

\section{The Hexagonal Torus and Ramanujan's Corresponding Theories}
In his 1914 article \cite{Ramanujan_Modular} Ramanujan established ``corresponding theories" to the theory of theta functions, however without proof (see also \cite[Chap.~33]{RamanujanV}). The theories Ramanujan anticipated, rely on Gauss' hypergeometric functions
\begin{equation}
	{}_2F_1 \left(\tfrac{1}{r}, \tfrac{r-1}{r};1; \, . \, \right),
\end{equation}
with $r = 2,3,4,6$. The theories are usually referred to as theory of signature 2, 3, 4 and 6, respectively \cite[Chap.~33]{RamanujanV}.

In \cite{Ramanujan_Collected} the editors give the following quote of Morell:``It is unfortunate that Ramanujan has not developed in detail the corresponding theories" (see also \cite[Chap.~5.5]{BorBor_Pi} and \cite[Chap.~33, Sec.~1]{RamanujanV}). The proofs were given much later by Borwein and Borwein \cite{BorBor_Pi}, \cite{BorBor_Cubic_91}. However, in \cite{BorBor_Pi}, after the proofs have been established, the authors state:``The explanation as provided by this section is a bit disappointing, since for all these theories, all we have are well-concealed versions of the original theory of $K$."\footnote{$K$ is the complete elliptic integral of the first kind.}

We will show that the theory of signature 3 is intimately connected to the hexagonal torus, just as we have seen that the theory of signature 2 is connected to the square torus (recall the results in Section \ref{sec_consequences_hk}). This may shed new light on Ramanujan's corresponding theories.

The generating matrix for the hexagonal lattice is given by
\begin{equation}
	S_h = \tfrac{\sqrt{2}}{\sqrt[4]{3}}
	\begin{pmatrix}
		1 & \tfrac{1}{2}\\
		0 & \tfrac{\sqrt{3}}{2}
	\end{pmatrix}.
\end{equation}
We will denote the hexagonal lattice by
\begin{equation}
	\L_h = S_h \Z^2.
\end{equation}
The name refers to the fact that its Voronoi cell \cite{ConSlo99} is a regular hexagon. Alternatively, it is sometimes called a triangular lattice as half of its fundamental domain, so to say its fundamental triangle, is an equilateral triangle.

The lattice theta function\footnote{In this section we will not distinguish between $\widetilde{p}_\L$ and $p_\L$.} which describes the heat kernel on the hexagonal torus $\T^2_{\L_h}$ is given by
\begin{align}
	p_{\L_h}(z;t) & = \tfrac{1}{t} \sum_{(k,l) \in \Z^2} e^{-\pi \tfrac{1}{t} \tfrac{2}{\sqrt{3}} \left((k+x)^2 + (k+x)(l+y) + (l+y)^2\right)}\\
	& = \sum_{(k,l) \in \Z^2} e^{- \pi t \tfrac{2}{\sqrt{3}} \left(k^2 + kl + l^2\right)} e^{2 \pi i (ky - lx)}, \quad z = (x,y) \in \T^2.
\end{align}
We note that we actually used the more explicit formula for the heat kernel on the standard torus $\T^2$ with hexagonal metric $S_h$. However, we will keep this notation in the sequel. Also, note the minus sign in the complex exponential, which is the result of using the symplectic version of the Poisson summation formula. Again, by using the triangle inequality, it is easy to show that
\begin{equation}
	p_{\L_h}(z;t) \leq p_{\L_h}(0;t), \qquad \forall z \in \T^2, \, t \in \R_+.
\end{equation}
This result actually holds for any lattice $\L$, not only for $\L_h$. Finding the minimal value of $p_\L(z;t)$ is in general a hard task as already remarked in \cite{Baernstein_HeatKernel_1997}. Numerical experiments also show that, in general, the location of the minimal value depends on $t$. Interestingly, this was not the case for rectangular tori and, also, it is not the case for the hexagonal torus;
\begin{equation}
	p_{\L_h}\left( \left( \tfrac{1}{3}, \tfrac{1}{3} \right); t\right) = p_{\L_h}\left( \left( \tfrac{2}{3}, \tfrac{2}{3} \right); t\right) \leq p_{\L_h}(z;t), \qquad \forall z \in \T^2, \, t \in \R_+.
\end{equation}
This is just an adaption of the result of Baernstein \cite{Baernstein_HeatKernel_1997} who showed that the minimal temperature on the hexagonal torus is taken at the barycenter of a fundamental (equilateral) triangle. It is the high symmetry of the rectangular and hexagonal lattices, which force the minimal value to  stay in one place for all time $t \in \R_+$.
\begin{figure}[ht]
	\subfigure[The heat distribution on a torus with standard metric.]{
		\includegraphics[width=.45\textwidth]{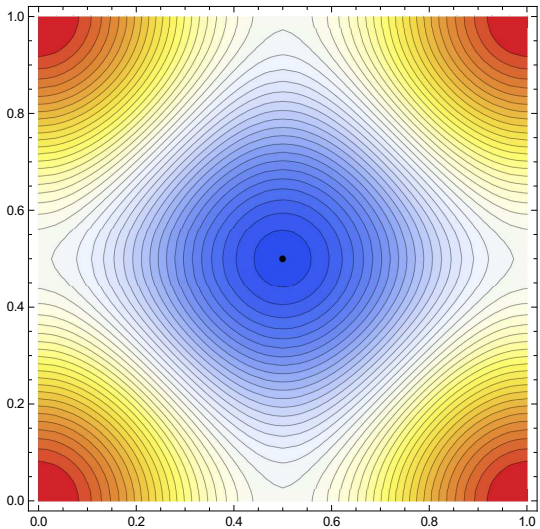}
	}
	\hfill
	\subfigure[The heat distribution on a torus with the hexagonal metric $S_h$.]{
		\includegraphics[width=.45\textwidth]{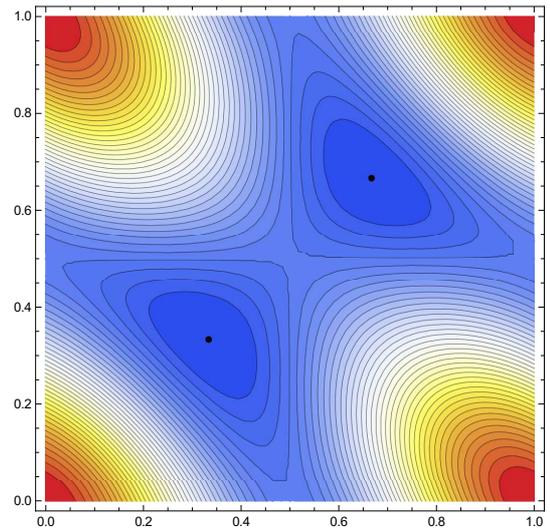}
	}
	\caption{Illustration of the heat distribution on a torus with standard metric and a torus with hexagonal metric for $t = 1$. The maximal temperature is taken in the corners, i.e., at lattice points. The position of the minimal temperature is marked.}
\end{figure}

We will now state some results analogous to Ramanujan's formula \eqref{eq_magic} for the hypergeometric function ${}_2F_1\left(\tfrac{1}{3},\tfrac{2}{3};1;\tfrac{1}{2}\right)$ involving cubic analogues of the squares of Jacobi's theta functions \cite{BorBor_Cubic_91}. The cubic analogues to the squares of Jacobi's theta functions, in the notation of \cite{BorBor_Cubic_91}, are the functions
\begin{equation}
	a(q) = \sum_{(k,l) \in \Z^2} q^{k^2 + k l + l^2},
\end{equation}
\begin{equation}
	b(q) = \sum_{(k,l) \in \Z^2} q^{k^2 + k l + l^2} e^{2 \pi i \left( \tfrac{k}{3} - \tfrac{l}{3} \right)}
	\qquad \textnormal{ and } \qquad
	c(q) = \sum_{(k,l) \in \Z^2} q^{\left(k+\tfrac{1}{3}\right)^2 + \left(k+\tfrac{1}{3}\right)\left(l+\tfrac{1}{3}\right) + \left(l+\tfrac{1}{3}\right)^2}.
\end{equation}
They fulfill
\begin{equation}\label{eq_cubic}
	a(q)^3 = b(q)^3+c(q)^3
\end{equation}
and by setting
\begin{equation}
	s = \frac{c(q)}{a(q)} \qquad \textnormal{ and } \qquad s' = \frac{b(q)}{a(q)}
\end{equation}
we have a cubic analogue to the squared (complementary) elliptic modulus. Also, we have
\begin{equation}\label{eq_magic_hex}
	{}_2F_1\left(\tfrac{1}{3},\tfrac{2}{3};1;s^3\right) = a(q),
\end{equation}
which was claimed already by Ramanujan \cite[Chap.~33]{RamanujanV} and then proven in \cite[Thm.~2.3.]{BorBor_Cubic_91}.

We note that the extremal temperatures on the hexagonal torus $\T_{\L_h}$ can be expressed by $a(q), \, b(q), \, c(q)$;
\begin{equation}
	A_{\L_h}(t) = b \left( e^{- \pi \tfrac{2}{\sqrt{3}} t} \right) = \tfrac{1}{t} \, c \left( e^{- \pi \tfrac{2}{\sqrt{3}} \tfrac{1}{t}} \right)
	\quad \textnormal{ and } \quad
	B_{\L_h}(t) = a \left( e^{- \pi \tfrac{2}{\sqrt{3}} t} \right) = \tfrac{1}{t} \, a \left( e^{- \pi \tfrac{2}{\sqrt{3}} \tfrac{1}{t}} \right).
\end{equation}
This shows the intimate connection of the theory of signature 3 and the hexagonal torus.

Furthermore, we note that for all the cases we treated in this article so far, we have root systems in the background. These are finite vector systems of high symmetry. For a finite dimensional vector space $V$, a finite collection of vectors, denoted by $R$, is called a root system if it fulfills the following properties;
\begin{enumerate}[(i)]
	\item The elements of $R$ do not contain 0 and span $V$.
	\item For any root $\alpha \in R$, the only scalar multiples of $\alpha$ contained in $R$ are $\alpha$ and $-\alpha$.
	\item For $\alpha, \beta \in R$, the set $R$ contains the element
	\begin{equation}	
		s_\alpha(\beta) = \beta - 2 \frac{\alpha \cdot \beta}{\alpha \cdot \alpha} \, \alpha.
	\end{equation}
	\item For $\alpha, \beta \in R$, we have
	\begin{equation}
		2 \frac{\alpha \cdot \beta}{\alpha \cdot \alpha} \in \Z.
	\end{equation}
\end{enumerate}
In $\R^2$ there are, up to isomorphy, 4 root systems, 2 of which generate the square lattice and 2 of them generate the hexagonal lattice \footnote{If one is a bit more exact, then one of the two root systems generating the square lattice is actually allowed to have generating vectors of different lengths (it is the tensor product of two one-dimensional root systems). Therefore, this root system also generates any rectangular lattice.}. For further reading we refer to the textbooks \cite[Chap.~6]{Bou02} and \cite[Chap.~8]{Hal15} and for the importance of root systems to special functions we refer to \cite{Mac71}. The theories with signature 2 and 4 each connect to one of the root systems generating the square lattice. The theories of signature 3 and 6 each connect to a root systems yielding hexagonal lattices.

Also, we note that we can write the heat kernel in dependence of $\tau \in \mathbb{H}$, defining a lattice in $\C$, i.e.,
\begin{equation}
	\L_\tau = \lbrace Im(\tau)^{-1/2} (k + \tau l) \mid k,l \in \Z , \, \tau \in \mathbb{H} \rbrace,
\end{equation}
where $Im(\tau)^{-1/2}$ normalizes the lattice to have area 1. The heat kernel with metric induced by $\tau$ is then given by
\begin{equation}
	p_\tau(z;t) = \tfrac{1}{t} \sum_{\l_\tau \in {\L_\tau}} e^{-\pi \tfrac{1}{t} \, |\l_\tau + z|^2} = \sum_{\l_\tau \in {\L_\tau}} e^{-\pi t \, |\l_\tau|^2} e^{2 \pi i \, Im\left( \overline{\l_\tau} \, z \right)}.
\end{equation}
The special choice $\tau = i$, gives the square torus and the sum is actually over the Gaussian integers. For the choice $\tau = e^{\pi i /3} = \tfrac{1+i\sqrt{3}}{2}$ we have the hexagonal torus and the sum is over the Eisenstein integers (scaled by a factor $\tfrac{\sqrt{2}}{3^{1/4}}$). Therefore, we can also refer to the cases as the lemniscatic case and the equianharmonic case \cite[Chap.~18]{AbrSte72}. The expert may also draw the connection to complex elliptic curves or, possibly relevant for analogous problems in higher dimensions, the connection to Lambert series as shown up in \cite{BorBor_Cubic_91}.

Last in this section, we note that the main result in Montgomery's article \cite{Montgomery_Theta_1988} is (equivalent to) the following statement;
\begin{equation}\label{eq_Montgomery}
	B_{\L_h}(t) \leq B_\L(t), \qquad \forall t \in \R_+,
\end{equation}
with equality if and only if $\L$ is a hexagonal lattice.

The question that remains open is whether or not the hexagonal torus uniquely maximizes the lowest temperature;
\begin{equation}\label{eq_conjecture}
	A_{\L_h}(t) \stackrel{(?)}{\geq} A_\L(t), \qquad \forall t \in \R_+,
\end{equation}
with equality if and only if $\L$ is a hexagonal lattice. This question has received less attention, but is of great interest in time-frequency analysis and signal reconstruction. In particular, the problem arises as a natural question in \cite{Faulhuber_PhD_2016} as a combination of Montgomery's result and the conjecture of Strohmer and Beaver on optimal Gaussian Gabor frames \cite{StrBea03}. The author conjectures that \eqref{eq_conjecture} indeed holds if and only if $\L$ is hexagonal.

\section{Landau's ``Weltkonstante" and a related Problem}
In this section we will briefly discuss the conjectured values of Landau's problem and the related problem formulated in \cite{Baernstein_Metric_2005} and \cite{Eremenko_Hyperbolic_2011}. A reformulation of Landau's problem is given in \cite{BaernsteinVinson_Local_1998} and is as follows. Let $\Gamma \subset \C$ be a (relatively separated) discrete set and consider the universal cover of $\C \backslash \Gamma$ by $\D$. In the case that $\Gamma$ is a lattice, this is the universal cover of the once punctured (complex) torus. Rademacher's conjecture on the precise value of $\LC$ can now be formulated as follows \cite{BaernsteinVinson_Local_1998}.
\begin{conjecture}\label{con_Rad}
	For each discrete subset $\Gamma \subset \C$, each universal covering map $f$ of $\D$ onto $\C \backslash \Gamma$, and each $z \in \C$ the following holds;
	\begin{equation}
		\frac{\left(1-|z|^2\right) |f'(z)|}{r(f)} \leq |f_h'(0)|.
	\end{equation}
\end{conjecture}
Here, $r(f)$ is the radius of the largest disc that can be placed in $f(\D) = \C \backslash \Gamma$ \footnote{Without loss of generality, here, we may assume that $f(0)$ is the center of the largest disc that can be placed in $\Gamma$.}, so it is the size of the covering radius of $\Gamma$, and $f_h$ is the universal cover of $\D$ onto the hexagonal torus with covering radius 1. We note that $|f_h'(0)| \approx 1.84074 \ldots$ which is, of course, the reciprocal of $\LC_+$ derived by Rademacher \cite{Rad43} and given in \eqref{eq_LC}. In fact, for the square torus and the hexagonal torus, an explicit construction of the universal covering map is possible. This fact was used by Rademacher as he explicitly computed $|f_h'(0)|$ in his article \cite{Rad43} by constructing the universal covering map of the hexagonal torus.

For the rest of the work we fix $t = 1$. With this restriction, we introduce the following universal constants arising from \eqref{eq_AB}, which are closely related to Landau's ``Weltkonstante";
\begin{align}
	\mathcal{A} = \mathcal{A}^*(1)
	\qquad \textnormal{ and } \qquad
	\mathcal{B} = \mathcal{B}_*(1).
\end{align}
Now, for $t = 1$, we derive the following values for the hexagonal torus;
\begin{align}
	A_{\L_h}(1) & = \sum_{(k,l) \in \Z^2} e^{-\pi \tfrac{2}{\sqrt{3}} \left( k^2 +k l+ l^2 \right)} e^{2 \pi i \left( \tfrac{k}{3} - \tfrac{l}{3} \right)} \approx 0.920371 \ldots\\
	B_{\L_h}(1) & = \sum_{(k,l) \in \Z^2} e^{-\pi \tfrac{2}{\sqrt{3}} \left( k^2 +k l+ l^2 \right)} \approx 1.159595 \ldots \, .
\end{align}
Note that by \eqref{eq_cubic} we have
\begin{equation}\label{eq_ratio}
	\frac{A_{\L_h}(1)}{B_{\L_h}(1)} = 2^{-1/3}.
\end{equation}

By Montgomery's result we know that
\begin{equation}
	\mathcal{B} = B_{\L_h}(1).
\end{equation}
For the lower temperature, the question remains whether the hexagonal torus yields the optimal solution;
\begin{equation}\label{eq_con_A}
	\mathcal{A} \stackrel{(?)}{=} A_{\L_h}(1).
\end{equation}

Now, assuming the correctness of \eqref{eq_con_A}, i.e., that the maximizer of the lowest temperature at time $t=1$ among all tori of area 1 with flat metric is the hexagonal torus, we get that, up to a factor of 2 \footnote{Obviously, the factor 2 would disappear if in Theorem \ref{thm_Landau} we asked for the diameter of the image to be 1 instead of its radius.}, this is the reciprocal of the conjectured value of Landau's ``Weltkonstante". By using the result claimed by Ramanujan \eqref{eq_magic_hex}, Gauss' formula \eqref{eq_Gauss} and formula \eqref{eq_ratio} we get
\begin{align}
	2 \mathcal{A} & = 2 \sum_{(k,l) \in \Z^2} e^{-\pi \tfrac{2}{\sqrt{3}} \left( k^2 +k l+ l^2 \right)} e^{2 \pi i \left( \tfrac{k}{3} - \tfrac{l}{3} \right)}\\
	& = 2^{2/3} \, {}_2F_1(\tfrac{1}{3}, \tfrac{2}{3}; 1; \tfrac{1}{2}) =
	\frac{\Gamma\left( \tfrac{1}{6} \right)}{\Gamma\left( \tfrac{1}{3} \right) \Gamma\left( \tfrac{5}{6} \right)} = |f'_h(0)| = \LC_+^{-1}.
\end{align}
In \cite[Sec.~7]{Faulhuber_PhD_2016} it was conjectured that the values $2 A_{\L_h}(1)$ and $\LC_+^{-1}$ yield the same constant, which we have now proven. Hence, the expected solution to Landau's problem would be given by
\begin{equation}
	\LC = \frac{1}{2	\mathcal{A}}.
\end{equation}
Also, in \cite{Faulhuber_PhD_2016} Rademacher's techniques were used to construct the universal covering map of the square torus of covering radius 1 (see also \cite[Chap.~VI, Sec.~5]{Nehari_Conformal_1975}) and it was proven that
\begin{equation}
	\LC_\square = \frac{1}{2 G},
\end{equation}
where, as already mentioned, $G = \theta_4(e^{-\pi})^2 \approx 0.834627 \ldots$ is Gauss' constant. Hence, the value of the separable Landau constant $\LC_\square \approx 0.59907 \ldots$ is exactly the second lemniscate constant.

Furthermore, we derive the exact value for the conjectured solution to the problem posed in \cite{Baernstein_Metric_2005} and \cite{Eremenko_Hyperbolic_2011}, which basically restricts the set $\Gamma$ in Conjecture \ref{con_Rad} to be a rectangular lattice. If we denote the universal covering map of the square torus of covering radius 1 by $f_\square$, then it is possible to explicitly compute (see \cite[Sec.~7]{Faulhuber_PhD_2016} and \cite[Chap.~VI, Sec.~5]{Nehari_Conformal_1975})
\begin{equation}
	|f'_\square(0)| = \frac{1}{\LC_\square} = \frac{\Gamma\left( \tfrac{1}{4} \right)}{\Gamma\left( \tfrac{1}{2} \right) \Gamma\left( \tfrac{3}{4} \right)} = \sqrt{2} \, {}_2F_1(\tfrac{1}{2}, \tfrac{1}{2}; 1; \tfrac{1}{2}) = 2 \, \theta_4(e^{-\pi})^2 = 2 G = 2 A^*(1),
\end{equation}
where we again used Gauss' formula \eqref{eq_Gauss} and Ramanujan's formula \eqref{eq_magic}.

Lastly, we note that it is easy to show that for any torus $\T^2_\L$ and arbitrary $t \in \R_+$ the lowest temperature fulfills
\begin{equation}
	A_\L(t) \leq 1,
\end{equation}
with equality as $t \to \infty$. This is, so to say, a uniform upper bound on the lowest temperature on the standard torus $\T^2_{(\Z^2,S)}$ with varying metric induced by $S \in SL(2,\R)$ over all $t \in \R_+$. Equivalently, we have
\begin{equation}
	\frac{1}{2} \leq \frac{1}{2 A_\L(t)},
\end{equation}
which goes along quite nicely with the general lower estimate on Landau's constant established by Ahlfors \cite{Ahl38} using ultrahyperbolic metrics, mentioned in the introduction;
\begin{equation}
	\frac{1}{2} \leq \LC.
\end{equation}

To conclude, it seems that if we are looking for a connection between Landau's problem and the heat kernel, we should not be looking at the minimal temperature of tori with fixed covering radius, but at the minimal temperature of tori of fixed area. The seemingly strange part of comparing universal covering maps of tori of fixed covering radius $k^2+k'^2 = 1$ with heat kernels on tori of fixed area 1 seems to be resolved by the elliptic modulus, at least in the rectangular case.

\bibliographystyle{plain}


\end{document}